\documentclass[11pt]{article}
\usepackage{amsthm}
\usepackage{amsmath,amsthm,amsfonts,amssymb,mathrsfs,bm,graphicx,stmaryrd,dsfont}
\usepackage[colorlinks=true,linkcolor=blue]{hyperref}
\usepackage[usenames,dvipsnames]{color}
\usepackage{fullpage}
\usepackage[american]{babel}
\usepackage[varg]{pxfonts}
\usepackage{prettyref}
\usepackage{microtype}

\usepackage[utf8]{inputenc}

\usepackage{amsthm, amssymb}
\usepackage{amsmath}

\usepackage[colorlinks=true,linkcolor=blue]{hyperref}

\newcommand{\girth}{\mathrm{girth}}
\newcommand{\R}{\mathbb{R}}
\newcommand{\eps}{\varepsilon}

\newcommand\blfootnote[1]{%
  \begingroup
  \renewcommand\thefootnote{}\footnote{#1}%
  \addtocounter{footnote}{-1}%
  \endgroup
}

\usepackage{fullpage}
\newtheorem{theorem}{Theorem}[section]
\newtheorem{lemma}{Lemma}[section]
\newtheorem{prop}{Proposition}[section]

\theoremstyle{definition}
\newtheorem{remark}{Remark}[section]
\title{High-girth near-Ramanujan graphs with localized
eigenvectors}
\author{Noga Alon \thanks{
Research supported in part by
NSF grant DMS-1855464, ISF grant 281/17, BSF grant
2018267
and the Simons Foundation. Email: nalon@math.princeton.edu}
\\Princeton University \and Shirshendu Ganguly \thanks{sganguly@berkeley.edu; partially supported by a Sloan Research Fellowship in mathematics and NSF Award DMS-1855688.}\\ UC Berkeley \and Nikhil Srivastava \thanks{nikhil@math.berkeley.edu; supported by NSF grant CCF-1553751.}\\UC Berkeley}
\date{}
\begin{document}
\maketitle
\begin{abstract}
	We show that for every prime $d$ and $\alpha\in (0,1/6)$, there is an
	infinite sequence of $(d+1)$-regular graphs $G=(V,E)$ with girth at least $2\alpha \log_{d}(|V|)(1-o_d(1))$, second adjacency matrix eigenvalue
	bounded by $(3/\sqrt{2})\sqrt{d}$, and many eigenvectors fully
	localized on small sets of size $O(|V|^\alpha)$. This strengthens the
	results of \cite{GS}, who constructed high girth (but not expanding)
	graphs with similar properties, and may be viewed as a discrete
	analogue of the ``scarring'' phenomenon observed in the study of
	quantum ergodicity on manifolds. Key ingredients in the proof are a
	technique of Kahale \cite{kahale} for bounding the growth rate of
	eigenfunctions of graphs, discovered in the context of vertex
	expansion and a method of Erd\H{o}s and Sachs for constructing
        high girth regular graphs.
\end{abstract}
\blfootnote{This work was partially completed while SG and NS were at the Simons Institute for the Theory of Computing, as part of the ``Geometry of Polynomials'' program.}

\section{Introduction}
We study the relationship between geometric properties of finite regular graphs, such as girth and expansion,
and localization properties of their Laplacian / adjacency matrix eigenvectors. This line of work was initiated by Brooks and Lindenstrauss,
who proved that the eigenvectors of high girth  graphs cannot be too localized in the following sense (in fact, they studied graphs with few short cycles, but we will state
the restriction of their results for high girth graphs for simplicity).
\begin{theorem}[\cite{bl}]\label{thm:bl}
	Suppose $G=(V,E)$ is a $(d+1)-$regular graph with adjacency matrix $A$.
	Then for any normalized $\ell_2$ eigenvector $v\in\R^V$
	of $A$ and $S\subset V$ with $\|v_S\|_2^2:=\sum_{x\in S}v^2_x\ge \varepsilon,$
	\begin{equation}\label{eqn:bl}|S|\ge \Omega_{d}(\eps^{2}d^{2^{-7}\eps^2\girth(G)}),\end{equation}
	where $\girth(G)$ denotes the length of the shortest cycle in $G$.
\end{theorem}
The recent work \cite{GS} improved \eqref{eqn:bl} to 
\begin{equation}\label{eqn:improvement}
	|S|\ge \frac{\eps d^{\eps \girth(G)/4}}{2d^2},
\end{equation}
under the same assumptions on $G$.  Moreover, given any $\varepsilon\in (0,1),$ they proved that for infinitely
	many $m\in\mathbb{N}$ there is a $(d+1)-$regular graph $G_m$ with $m$ vertices, $\girth(G_m)\ge (1/8)\log_d(m)$, and  a localized eigenvector satisfying $\|v_S\|_2^2=\varepsilon$ and $|S|\le O(d^{4\varepsilon\girth(G_m)})$. 
	This shows that \eqref{eqn:improvement} is sharp up to constants and a factor of $\varepsilon d^{-2}$ in the 
	regime where the girth is logarithmic in the number of vertices.

In this work, we construct examples which improve on the above in three ways: (1) the graphs we construct are expanders with near-optimal
spectral gap, (2) we improve the bounds on $\girth(G_m)$ as well as the localization size $|S|$ by constant factors, (3) our constructions
are explicit whereas \cite{GS} used the probabilistic method to show existence non-constructively.
\begin{theorem}
\label{thm:totalloc} 
For every $d=p+1$, $p$ prime, and parameter $\alpha\in (0,1/6)$ there are infinitely many integers $m$ such that there exists a  $(d+1)-$regular graph $G_{m}=(V_{m},E_{m})$ on $m$ vertices with the following properties,
	\begin{enumerate}
		\item $|\lambda_i(A_{m})|\le (3/\sqrt{2})\sqrt{d}$ for all nontrivial eigenvalues $i\neq 1$ of the adjacency matrices $A_{m}$.
		\item $\girth(G_{m})\ge
2\alpha\log_{2d-1}(m)-O(1) = 2\alpha\log_d(m)\cdot (1-O(\log^{-1}(d))).$
		\item There is a set $S_{m}\subset V_{m}$ of size  $O(m^{\alpha})$ such that $A_{m}$ has at least $\ell_m:=\lfloor\alpha\log_d(m)\rfloor$ eigenvalues
			$\lambda\in (-2\sqrt{d},2\sqrt{d})$ with corresponding eigenvectors $v:A_{m}v=\lambda v$ supported entirely on $S_{m}$.
\end{enumerate} 
	{Moreover, the set of eigenvalues $\lambda$ realized by the localized eigenvectors of $A_{m},$ over all such $m$ is dense in the interval $(-2\sqrt{d},2\sqrt{d})$.}
\end{theorem}
Note that the number $3/\sqrt{2}\cdot\sqrt{d}\approx 2.121\sqrt{d}$ above is quite close to the best possible bound of $2\sqrt{d}$ for an infinite sequence of regular graphs \cite{alonboppana}.
\begin{remark}\label{partial}(Partial Localization on Smaller Sets).
In fact the proof of Theorem \ref{thm:totalloc} produces eigenvectors $v,$ with the additional property that for $\varepsilon\in (0,1)$, there exists a subset $S$ of vertices with $|S|=O(m^{\Theta(\varepsilon) \alpha}),$ and $\|v_S\|_2^2\ge \Theta(\varepsilon)$. 
\end{remark}
Finally, we show how to modify our construction to produce many localized eigenvectors corresponding to eigenvalues with very high multiplicity.
\begin{theorem}[Many Localized Eigenvectors]\label{thm:many} Theorem \ref{thm:totalloc} is true with the last property replaced by: 
	there are $\ell_m:=\alpha \log_d(m)$ eigenvalues $\lambda_1,\ldots,\lambda_{\ell_m}$, each of multiplicity
	at least $\Omega(m^{1-4\alpha})$, such that each eigenspace has a basis of orthogonal eigenvectors supported on sets of size
	$O(m^{\alpha})$.
\end{theorem}

\subsection{Implications for Quantum Ergodicity on Graphs}
The additional property of expansion in our examples is relevant to the study of quantum ergodicity on graphs. 
Anantharaman and Le Masson proved that if a sequence of graphs has 
few short cycles {\em and} a spectral gap, then the eigenvectors must be equidistributed on average in a sense stronger
than Theorem \ref{thm:bl}.
\begin{theorem}\label{thm:ananth}\cite{anantharaman, brooks2015} Suppose $G_m=(V_m,E_m)$ is a sequence of $(d+1)$-regular graphs on $m$ vertices with adjacency
	matrices $A_m$ satisfying:
	\begin{enumerate}
	\item [(BST)]  The sequence of graphs converges to a tree in the sense of
Benjamini-Schramm, i.e., there exist $R_m \to \infty$ and $\alpha_m \to 0$, such that  
		$$\frac1m |\{v\in V_m: N_m(v,R_m) \textrm{ contains a cycle}\}|\le \alpha_m,$$
			where $N_m(v,R)$ is the set of vertices at distance at most $R$ from $v,$ in $G_m$. Note that
			this condition is implied by $\girth(G_m)\rightarrow\infty$. 
	\item [(EXP)] There is a constant $\beta>0$ such that $$|\lambda_i(A_m)|<(d+1)(1-\beta),$$ for all nontrivial eigenvalues $i\neq 1$.
	\end{enumerate}
	Then for any sequence of test functions $a_m:V_m\rightarrow \mathbb{R}$ with $\sum_{x\in V_m}a_m(x)=0, \|a_m\|_\infty\le 1$:
	\begin{equation}
		\label{eqn:que}\frac{1}{m}\sum_{i\le m} |\langle \psi_i^{(m)}, a_m \psi_i^{(m)}\rangle|^2 \lesssim 
\beta^{-2} \min \{R_m, \log(1/\alpha_m)\}^{-1}
\|a_m\|_2^2+ \alpha_m^{1/2}\|a_m\|_{\infty}^2
\longrightarrow 0,
	\end{equation}
	where $\psi_1^{(m)},\ldots \psi_m^{(m)}$ is any eigenbasis of $A_m$.
\end{theorem}
The above may be viewed as a discrete analogue of 
the quantum ergodicity theorem of Shnirelman, Zelditch, and Colin de Verdi\'ere \cite{shnirel1974ergodic, de1985ergodicite, zelditch1987uniform}, which 
states that if the geodesic flow on a compact manifold is ergodic, then it must have a dense
subsequence of Laplacian eigenfunctions whose mass distribution converges weakly
to the volume measure as the energy goes to infinity. 
In Theorem \ref{thm:ananth}, the manifold has been replaced by a { sequence}
of graphs, the condition of ergodic geodesic flow has been replaced by { BST
and EXP}, and the notion of weak convergence involves a { sequence} of test
functions on the graphs rather than a single test function on the manifold.

An even  stronger notion of delocalization for the Laplacian on a manifold is
Quantum Unique Ergodicity (QUE) (see e.g. \cite{sarnaksurvey} for a detailed
discussion), where instead of a dense subsequence of eigenfunctions, one
requires that {\em every} subsequence of eigenfunctions becomes
equidistributed. 
It is not completely clear what the correct analogous notion should be for a
sequence of finite graphs.  There are various proposals; one definition which
appears in Anantaraman's ICM survey \cite{anantharaman2018delocalization} and in \cite[Question 1.3]{le2017quantum} in the context of sequences of
manifolds is: for every sequence of test functions $a_m$ as in Theorem
\ref{thm:ananth}, and {\em every} sequence of eigenfunctions
$\psi_{i_m}^{(m)}$, one has
\begin{equation}\label{eqn:quedef}
	|\langle \psi_{i_m}^{(m)}, a_m \psi_{i_m}^{(m)}\rangle| \longrightarrow 0.
\end{equation}

Since the graphs constructed in Theorem \ref{thm:totalloc} satisfy BST and EXP,
the theorem shows that these properties cannot imply unique ergodicity in the
above sense: take the $\psi_{i_m}^{(m)}$ to be the localized eigenvectors of
 $G_m$, and let $a_m$ be the indicator functions of the sets $S_m$ on which
they are localized, translated by a constant to have mean zero. It is then
immediate that $\langle \psi_{i_m}^{(m)}, a_m
\psi_{i_m}^{(m)}\rangle=1-o(1)$ for
the entire sequence.

The presence of localized eigenvectors is sometimes referred to as ``scarring''
(see e.g. \cite{anantharaman2018delocalization, hassell}), which may be partial or complete depending
on whether a large fraction or all of the mass is localized on a small set.
Theorem \ref{thm:totalloc} and Remark \ref{partial} may be
interpreted as saying that scarring can occur even under strong expansion and
girth assumptions.

\begin{remark}[QE over intervals] The works \cite{anantharaman, brooks2015} also
	study a more refined version of quantum ergodicity on graphs, where the
	average \eqref{eqn:que} is taken over a spectral window $I\subset
	(-2\sqrt{d},2\sqrt{d})$ rather than the entire spectrum. These results
	hold on intervals $I$ of width roughly $1/\log(m)$, and it would be
	interesting to see whether our examples can prove a lower bound on the
	length of the smallest window that is possible. While Theorem \ref{thm:many} does
	produce many localized eigenvectors in a very small window (due to high
	multiplicity), the problem of controlling the other
eigenvectors well
	enough to say that the average in a small window is not equidistributed is not pursued here and remains open.
\end{remark}

\subsection{Techniques and Vertex Expansion}
The starting point of the proofs of Theorems \ref{thm:totalloc} and \ref{thm:many} is a construction in the proof of  \cite[Theorem 1.6]{GS} which has the following ingredients:
\begin{enumerate}
\item (Pairing trees \cite[Lemma 3.4]{GS}): A pair of trees is glued by
	randomly identifying their leaves, ensuring that the final graph has high girth.
\item (Degree-Correction \cite[Lemma 3.5]{GS}): The above gluing yields
	an irregular graph where the identified leaf vertices have degree two.
		Each such vertex is identified
		with a particular vertex of degree $d-1$ in a degree-correcting
		gadget whose remaining vertices have degree $d+1$ thereby
		yielding a $d+1$ regular graph. 
\end{enumerate}
We modify this proof in two ways. First, we replace the random pairing in step
(1) by a more efficient, simpler, and 
deterministic method. Second, in order to obtain the
additional property of expansion, we replace the degree-correcting gadget in
step (2) by a high girth Ramanujan graph \cite{LPS,M}. To analyze the spectrum
of the resulting graph, we must argue that its largest nontrivial eigenvector
cannot have too much mass on the interface between the trees and the Ramanujan
graph --- once this is established, it is easy to analyze the contributions
from the two pieces separately.  We do this by employing a lemma of N. Kahale,
which supplies a way to control the mass of
eigenvectors on certain highly symmetric
sets (such as our interface) by exhibiting certain appropriate 
super-harmonic test
functions, and by a careful construction of such a function.

Kahale's lemma originally appeared in the influential paper \cite{kahale} which
showed that a $(d+1)$-regular graph $G=(V,E)$ with all nontrivial eigenvalues
bounded by $2\sqrt{d}+o_n(1)$ must have linear expansion at least
$(d+1)/2-o_n(1)$, where linear expansion is defined as: $$ \max_{S\subset V,
|S|=\gamma |V|} \frac{N(S)}{|S|},$$ for a small constant $\gamma>0$ 
(in fact,
he showed a more general inequality relating the parameters). As we discuss
in Remark \ref{rem:vertex}, this implies that our examples cannot have
$|\lambda_i|\le 2\sqrt{d}+o_n(1)$ since our gluing procedure produces a set
with significantly smaller linear vertex expansion than $(d+1)/2$.

Note that it is possible to prove Theorem \ref{thm:totalloc}
with a weaker bound of $3\sqrt{d}$ without using Kahale's lemma; however, since
the bound we attain is quite close to optimal and we have not seen this
technique appear in the quantum ergodicity literature, we believe it is
valuable to present it.

\section{Pairing trees}
Our goal is to construct high girth almost-Ramanujan expanders with one
or many localized eigenvectors. The starting point of the construction
is the following lemma, improving the one from \cite{GS} and
simplifying its proof. We refer to a finite tree in which all vertices except the leaves have degree $(d+1)$ and every leaf
is at distance $D$ from the root as a {\em d-ary tree of depth $D$}.

\begin{lemma}[Pairing of Trees]
\label{l11} 
Suppose $T_1$ and $T_2$ are two $d-$ary trees of depth $D$, each with
$n=(d+1)d^{D-1}$ leaf vertices $V_1$ and $V_2$. Then 
there is a bijection $\pi:V_1\rightarrow V_2$
such that the graph obtained from the vertex disjoint union of
$T_1$ and $T_2$
by identifying $v$ and $\pi(v)$ for all $v$
has girth at least 
$$
\lfloor 2 \log_{2d-1} (n-1) \rfloor+2~(~> 2 \log_{2d-1} n).
$$ 
\end{lemma}

\begin{proof}
We apply a variant of the method of Erd\H{o}s and Sachs \cite{ES},
(see also \cite{ABGR} for a similar argument).
Let $g$ be the maximum possible girth of a graph obtained as
above, and let $\pi:V_1\rightarrow V_2$ be a bijection for which
the girth is $g$ and the number of cycles of length exactly $g$ is
minimum. Note that $g$ is even, as the graph is bipartite.
Let $G$ be the graph with the identified leaves 
obtained by $\pi$, and let $L$ denote the set of all $n$
vertices of degree $2$ in it, that is, all the identified leaves.
Obviously every cycle of $G$ must contain vertices of $L$. 
Let $x \in L$ be a vertex 
contained in a shortest cycle $C$ of $G$. 
\vspace{0.1cm}

\noindent
{\bf Claim:}\ For every $k \geq 0$ the number 
of vertices $y \in L$ of distance at most
$2k$ from $x$ is at most $(2d-1)^k$.
\vspace{0.1cm}

\noindent
{\bf Proof of claim:}\, Any shortest path of length precisely $2s
\leq 2k$ 
between $x$ and another
vertex $y \in L$ is a concatenation of some number $r$  of paths
$P_1,P_2, \ldots ,P_r$, where $P_i$ is a path from $x_{i-1}$ to
$x_i$ with $x_0,x_1, \ldots ,x_r \in L$, $x_0=x$,  $x_r=y$,
and either all even paths $P_i$ are in $T_1$ and all odd ones are
in $T_2$ or vice versa. Let $2k_i$ be the length 
of $P_i$, then $\sum_{i=1}^r k_i = s \leq k$. Let $m(r,s)$
denote the number of paths as above with these values of $r$ and $s$.
We next show that for all $1 \leq r \leq s$
\begin{equation}
\label{e11}
m(r,s)=2 {{s-1} \choose {r-1}} (d-1)^r d^{s-r}
\end{equation}
The factor $2$ is for deciding if the first path $P_1$ is in 
$T_1$ or in $T_2$. The factor ${{s-1} \choose {r-1}}$ is the number of
ways to choose the subset of $(r-1)$ elements $\{k_1,k_1+k_2,
\ldots k_1+k_2+  \cdots + k_{r-1}\}$ 
in the set $\{1,2, \ldots,s-1\}$.
(This already determines $k_r=s-(k_1+k_2 + \cdots + k_{r-1})$.) Once
these choices are fixed, there is only one way for the edges
numbers $1,2, \ldots ,k_i$ of each path $P_i$, given the previous
paths, as these edges go up
the tree. There are $d-1$ possibilities for the edge number $k_{i+1}$ of this
path, and there are $d^{k_i-1}$ choices for the remaining edges of
$P_i$. The product of all these terms gives the expression in
(\ref{e11}) for $m(r,s)$. For each fixed $s$, summing 
over all $1 \leq r \leq s$ we conclude that the number $m(s)$ of 
paths of length exactly $2s$ starting in $x$ is
$$
m(s) =\sum_{r=1}^s 2 {{s-1} \choose {r-1}} (d-1)^r d^{s-r}
=2(d-1) \sum_{j=0}^{s-1} {{s-1} \choose j} (d-1)^j d^{s-1-j} 
=2(d-1) (2d-1)^{s-1}.
$$
Adding the trivial path of length $0$ from $x$ to itself and summing
over all $s$ from $1$ to $k$ we conclude that
the total number of paths as above of length at most $2k$ starting
at $x$ is  
$$
1+\sum_{s=1}^k 2(d-1) (2d-1)^{s-1} =1+2(d-1)\frac{(2d-1)^k-1}{2d-2}
=(2d-1)^k.
$$
The number above provides an upper bound for the number of vertices $y
\in L$ that lie within distance $2k$ of $x$ (which may be smaller as
several paths may lead to the same vertex). This completes the proof
of the claim.

Returning to the proof of the lemma, define
$k = \lfloor \log_{2d-1} (n-1) \rfloor$. Then $(2d-1)^k <n$ and hence
there is a vertex $y \in L$ whose distance from $x$ in $G$ is larger
than $2k$ (and hence at least $2k+2$).  Let $u$ be the unique parent
of $x$ in $T_1$ and let $u'$ be the unique parent of $x$ in $T_2$.
Similarly, let $v$ be the unique parent of $y$ in $T_1$ and let $v'$
be the unique parent of $y$ in $T_2$. Change the bijection 
$\pi$ to the bijection $\pi'$ obtained by swapping the images of
$x$ and $y$ to get a graph $G'$ obtained from $G$ by removing the
edges $xu$ and $yv$ and by adding the edges $xv$ and $yu$. 
This swapping removes the shortest cycle $C$ of length $g$ in $G$
that contains $x$,
and is not contained in 
$G'$. Every new cycle contained in $G'$ and not in $G$ must include
at least one of the new edges $xv$, $yu$. If it contains exactly one
of them, say $xv$, then it must also contain a path in $G$ from $x$ to
$v$. The length of such a path is at least $2k+1$ (as the
distance in $G$ from $x$ to $y$ is at least $2k+2$) showing that in
this case the length of the new cycle is at least $2k+2$. If it
contains both new edges $xv$ and $yu$ it must also contain either a
path in $G$ from $x$ to $y$ (of length at least $2k+2)$) or a path in $G$ 
from $x$ to $u$ (of length at least $g-1$) and a path in $G$ from
$y$ to $v$ (of length at least $g-1$). Therefore, in this case the
length of the new cycle is either at least $2k+2+2=2k+4$ (in fact
larger) or at least
$2g-2+2=2g>g$.  It follows that the only possibility to obtain a new
cycle of length at most $g$ is if $g \geq 2k+2$. If the girth $g$ is
smaller than $G'$ has girth at least $g$ and the number of its cycles
of length $g$ is smaller than that number in $G$, contradicting the
choice of $G$. This shows that the girth $g$ satisfies
$g \geq 2k+2= 2 \lfloor \log_{2d-1} (n-1) \rfloor+2$, completing the
proof of the lemma.
\end{proof}

\begin{remark}
{For large fixed $d,$ the above graph has girth close to
$2 \log_d N$, where $N$ is the number of its vertices. This is 
strictly larger than the highest known girth of an
$N$ vertex $(d+1)$-regular graph, which is roughly
$\frac{4}{3} \log_d N$ (for some values of $d$). However,
many of the vertices of our graph here
have degree $2$, and suppressing them will not give graphs
of girth larger than $(1+o(1)) \log_d N$.}
\end{remark}
\section{The construction}\label{construction}

Let $d$ a prime and $\alpha\in (0,1/6)$ be given. 
Let $H=(V,E)$ be a $(d+1)$-regular non-bipartite 
Ramanujan graph with $m$ vertices and girth larger than $2/3\log_d(m)$; by
\cite{LPS} such graphs exist for infinitely many $m$. Set $r=\lfloor \alpha \log_d(m)\rfloor$ and note that
\begin{equation}\label{eqn:girthm}
 \frac{2}{3} \log_d m   \geq  4r.
\end{equation}

Fix a vertex $u$ of $H$. The induced
subgraph on all vertices of distance at most $r$ from $u$ 
is a tree $T_1$ rooted at $u$.
{Let $n$ be the number of its leaves, let the set of
leaves be $L_1=\{u_1,..,u_n\}$ and let $V_1$ denote the set of 
all non-leaves of $T_1$. Take a matching from the set $L_1$  
to the set of vertices $L_2=\{v_1,..,v_n\}$,
all at distance exactly $r+1$ from $u$, 
and remove the matching $u_iv_i$. Note that
all $u_i$ are far from each other in the graph $H-V_1$
since the girth is
significantly larger than $2r$. Similarly, all
vertices $v_i$ are far from each other in $H-V_1$ 
for the same reason.  Now take another $d$-tree $T_2$ isomorphic
to $T_1$ on new vertices, and let $u'$ denote its root.
Identify the leaves of $T_1$ with these of $T_2$
using Lemma \ref{l11}. Let $V_2$ denote the set of all non-leaves
of $T_2$.
As the vertices 
$u_i$ are far from each other in $H-V_1$ the
girth stays as large as guaranteed by the lemma. Finally add
a
third tree $T_3$ with the same parameters on new vertices, rooted
at $v'$, 
identify its leaves with the vertices $v_i$ and let
$V_3$ denote the set of its non-leaves. Call the resulting
graph $G$. }

The next sequence of lemmas will be needed to show that $G$ satisfies the claims of Theorem \ref{thm:totalloc}.

\begin{lemma}
\label{girth} The girth of $G$ is at least 
$$2\log_{2d-1}((d+1) d^{r-1}) ~(\geq 2 \alpha \log_{2d-1} (m)
-O(1)).
$$
\end{lemma}
\begin{proof}Follows from the above definition, 
Lemma \ref{l11} and that  $\frac{2}{3} \log_d m  > 4r$. 
\end{proof}

We now discuss the eigenvectors of $G$. We begin by recording some facts about eigenvalues and eigenvectors of  {rooted} $d-$ary
trees which also appear in \cite{GS} and will be critical to our construction. Recall that the eigenvalues of a $d-$ary tree are contained in the
interval $(-2\sqrt{d},2\sqrt{d})$ \cite[Section 5]{hoory}. For our purposes we will
only consider eigenvalues corresponding to eigenvectors which are {\em radial}, which means that they assign the same value to
vertices in a given level. We will refer to such eigenvalues as {\em radial eigenvalues}.

\begin{lemma}(Radial Eigenvalues)\cite[Lemma 3.1]{GS}\label{sym1} For any positive integer $D\ge 2,$ $A_D$ the adjacency matrix of $T_D,$ a $d-$ary tree of depth $D,$ has exactly $D+1$ radial eigenvalues counting multiplicities. 
\end{lemma}

\begin{lemma}(Eigenvalues of $d-$ary Trees)\cite[Lemma 3.2]{GS}\label{lem:dense} The set of all  {radial eigenvalues}
	 of any infinite sequence of distinct finite $d-$ary trees is dense in the interval
	$(-2\sqrt{d},2\sqrt{d})$.
\end{lemma}

\begin{lemma}[Eigenvectors of $d-$ary Trees]\cite[Lemma 3.3]{GS}\label{lem:eigvec} Assume $d\ge 2$
	and let $T$ be a $d-$ary tree of depth $D$ with root $r$. Let $S_0=\{r\},S_1,\ldots,S_D\subset
	T$ be the vertices at levels $0,1,\ldots, D$ of the tree and let $v$ be a
	radial eigenvector of its adjacency matrix with eigenvalue 
	$\lambda=2\sqrt{d}\cos\theta\in (-2\sqrt{d},2\sqrt{d})$. Then every
	pair of adjacent levels has approximately the same total $\ell_2^2$ mass as the root:
	$$\Omega(\sin^2\theta) = \frac{\|v_{S_i}\|_2^2+\|v_{S_{i+1}}\|_2^2}{\|v(r)\|_2^2}=O(1/\sin^{2}\theta).$$
\end{lemma}

Given the above we have the following lemma about how radial tree eigenvectors can be used to construct eigenvectors of $G.$

\begin{lemma}\label{global} For any radial eigenvalue $\lambda$ of the  adjacency matrix of a tree of depth $r-1$, there exists an eigenvector $\nu$ supported on $V_1\cup V_2$ such that $$A_G\nu=\lambda \nu.$$ 
\end{lemma}

\begin{proof} For completeness we include the arguments that essentially appear in \cite[Proof of Theorem 1.6]{GS}. Consider any such eigenvalue $\lambda$ and its corresponding radial eigenvector $f$. Now construct the function $\nu$ that equals $f$ on the top $r-1$ levels of $T_1$, i.e., $V_1$ and correspondingly $-f$ on $V_2$, and is zero elsewhere. 
 We claim that $\nu$ is an eigenvector of $G$ with eigenvalue $\lambda$.  To see this, note that the eigenvector equation is trivially satisfied on $V_1$ and $V_2$ because all new neighbors of those vertices are assigned a value of $0$ in $\nu$. The remaining  vertices where the eigenvector equation needs to be checked are the ones obtained by gluing $L_1$ to the leaves of $T_2.$ Now  every such vertex $v,$ satisfies $\nu(v)=0$ and there exists two neighbors of $v$, say $u\in V_1$ and $w\in V_2$ with $\nu(u)=-\nu(w)$ and furthermore $\nu$ is $0$ on every remaining neighbor, clearly implying the eigenvector equation at $v.$  	 
\end{proof}

We next show that $G$ is nearly Ramanujan. 

\section{The spectrum of $G$}
\begin{prop}
\label{t31}
For every fixed $\varepsilon >0$, if $m$ is sufficiently large then
the absolute value of every nontrivial eigenvalue of $G$ is at most
$$ (\frac{3d-1}{\sqrt {d(2d-1)}} +\varepsilon) \sqrt d.$$ 
\end{prop}
\begin{remark}\label{rem:vertex}
For every fixed $d,$ the number 
$\frac{3d-1}{\sqrt {d(2d-1)}}$ is
smaller than 
$\frac{3}{\sqrt 2} = 2.12132..$  
For $d=2$ (cubic graphs) the number is $5/\sqrt 6 =2.04124.. $. 
Therefore, the graph $G$ is close to being Ramanujan.\\
	However, for all
$r>1$ it is not quite Ramanujan. Indeed, it contains a set
of vertices $Y$, namely the set of all vertices in levels
$r-1,r-3,r-5,...$ of the two trees $T_1,T_2$, that expands by a
factor of less than $(d+1)/2$. {By Theorem 4.1 in \cite{Ka}} such
a graph must contain a nontrivial eigenvalue of absolute value
bigger than $2 \sqrt d+\delta(d,r)$ for some positive
$\delta(d,r)$.
\end{remark}
To prove Proposition \ref{t31}, we need the following simple lemma about the spectrum of finite
$d$-ary trees.
\begin{lemma}
\label{l32}
Let $T$ be a finite $d$-ary tree, that is, a tree with a root $r$ of
degree $d+1$ in which every non-leaf has $d$ children. Let
$A_T$ denote the adjacency matrix of $T$, let
$W$ be the set of its non-leaves 
and let $L$ be the set of its leaves. Then for
any vector $f$ supported on $W \cup L$,
$$
|f^T A_{T}f|\le 2\sqrt{d} \sum_{w \in W} f(w)^2 +
\sqrt{d} \sum_{v \in L} f(v)^2.
$$ 
\end{lemma}

\begin{proof} Orient all edges towards the root
$r$. Then 
\begin{align*}
|f^T A_{T}f|&=\sum_{u\to v}2|f(u)||f(v)|\\
&= \sum_{u\to v}2t |f(u)|\frac{|f(v)|}{t} \text{ for any choice of
$t>0$}\\
&\le \sum_{u\to v}\left(t^2f^2(u)+ \frac{1}{t^2}f^2(v) \right)\\
	&= \frac{d+1}{t^2}f^2(r)+t^2\sum_{u\in L}f^2(u)+\sum_{u\in
W-\{r\}}f^2(u)\left(t^2+\frac{d}{t^2}\right)\\ 
&\le2\sqrt{d}\sum_{u\in W }f^2(u) +\sqrt d \sum_{u \in L} f^2 (u) 
\text{ by choosing
$t^2=\sqrt{d}$}
\end{align*}
\end{proof}

We also need the following lemma of Kahale about growth rate of eigenfunctions.
\begin{lemma}(\cite[Lemma 5.1]{Ka}) 
\label{l33}
Consider a graph on a vertex set $U$, and let $A$ denote its
adjacency matrix.
Let $X$ be a set of vertices. 
Let $h$ be a positive integer and let $s$ be a function
on $U$. Let $X_i$ be the set of all vertices at distance $i$ from
$X$, and assume that the following conditions hold.
\begin{enumerate}
\item
For $h-1 \leq i,j \leq h$ all vertices in $X_i$ have the same
number of neighbors in $X_j$.
\item
The function $s$ is constant  on $X_{h-1}$ and on $X_h$.
\item
The function $s$ is positive and $As(v)  \leq |\mu| s(v)$ for every
$v $ of distance at most $h-1$ from $X$, where
$\mu$ is a nonzero real.
\end{enumerate}
Then for any function $g$ on $U$ which satisfies
$|Ag(u)|=|\mu| |g(u)|$ for all vertices $u$ of distance at most
$h-1$ from $X$, we have:
$$
\frac{\sum_{v \in X_h} g(v)^2}{\sum_{v \in X_h} s(v)^2}
\geq 
\frac{\sum_{v \in X_{h-1}} g(v)^2}{\sum_{v \in X_{h-1}} s(v)^2}.
$$
\end{lemma}

Equipped with the above lemma we now proceed to proving Proposition \ref{t31}.
\begin{proof}[Proof of Proposition \ref{t31}]

The adjacency matrix $A_G$ of $G$  is the sum
$A_G=A_H-A_M+A_{T_2}+A_{T_3}$, where $A_H$ is the adjacency matrix
of $H$,  $A_M$ is the adjacency matrix of the matching $u_iv_i$ and
$A_{T_2}, A_{T_3}$ are the adjacency matrices of the trees 
$T_2$ and $T_3$, respectively. It is convenient  to view all the
graphs $H,T_2,T_3,M$ as graphs on the set of all vertices of
$G$, which is $U=V \cup V_2 \cup V_3$. Thus all the matrices above
have rows and columns indexed by the  set $U$, and each of the
corresponding graphs has many isolated vertices.

Put $b=\frac{3d-1}{\sqrt {d(2d-1)}} $.
We have to show that every nontrivial eigenvalue $\mu$ of $G$
has absolute value at most $(b+\varepsilon)\sqrt d$ for any
$\varepsilon>0$ provided $m$ is sufficiently large. 
Let $g: U \rightarrow R$
be an eigenvector of $\mu$ satisfying $\sum_{v \in U} g(u)^2=1$.
As $g$ is orthogonal to the top eigenvector,
$\sum_{v \in U} g(v)=0$. 
The total number of vertices in $V_2 \cup V_3$ is smaller
than $2n$, and therefore, by Cauchy-Schwartz, $|\sum_{v \in V_2
\cup V_3} g(v)| \leq \sqrt {2n}$. It thus follows that
$|\sum_{v \in V} g(v)| \leq \sqrt {2n}.$ Considering the projection
of the restriction of $g$ to $V$ on the  all ones vector and its
complement we conclude that
\begin{equation}
\label{e34}
|g^T A_H g| \leq 2 \sqrt d \sum_{v \in V} g^2(v)
+\frac{(d+1)2n}{m}.
\end{equation}
Recall that $L_1$ is the set of leaves of $T_2$ (and $T_1$). By
Lemma \ref{l32} 
\begin{equation}
\label{e35}
|g^T A_{T_2} g| \leq 2 \sqrt d \sum_{v \in V_2} g^2(v)
+\sqrt d \sum_{v \in L_1} g^2(v).
\end{equation}
Similarly
\begin{equation}
\label{e36}
|g^T A_{T_3} g| \leq 2 \sqrt d \sum_{v \in V_3} g^2(v)
+\sqrt d \sum_{v \in L_2} g^2(v).
\end{equation}
The contribution of the omitted matching can be bounded as follows
\begin{equation}
\label{e37}
|g^T A_{M} g| = |\sum_{i=1}^n 2 g(u_i)g(v_i)| \leq
\sum_{v \in  L_1 \cup L_2} g^2(v).
\end{equation}
Combining (\ref{e34}), (\ref{e35}), (\ref{e36}), (\ref{e37})
we conclude that
\begin{align}\label{keybound}
|\mu|=|g^T A g| &\leq 2 \sqrt d \sum_{v \in U} g^2(v)
+ (\sqrt d+1) \sum_{u \in L_1 \cup L_2} g^2 (u)
+\frac{(d+1)2n}{m},\\
\nonumber
&=2 \sqrt d + (\sqrt d+1) \sum_{u \in L_1 \cup L_2} g^2 (u)
+\frac{(d+1)2n}{m}.
\end{align}
In order to complete the proof, it thus suffices to show that 
if $|\mu| \geq   (b+\varepsilon) \sqrt d$, then, as $m$ tends
to infinity, the sum $\sum_{u \in L_1 \cup L_2} g^2 (u)$ tends to
zero. This is done using Lemma \ref{l33}, as described next.\\

\begin{enumerate} 
\item We first bound the sum $\sum_{u \in L_2} g^2(u)$. This is simple
and works even if we only assume that $|\mu| \geq 2 \sqrt d$.
Indeed, starting with $X=\{v'\}$, let $X_i$ be the set of vertices
of distance $i$ from $X$. Define $s(v)=d^{-i/2}$ for all $v \in
X_i$ for $0 \leq i \leq r+t$, where $r+t$ is the largest
integer smaller than half the girth of $H$. It is easy to check
that the conditions of Lemma \ref{l33} hold. Thus by its conclusion
the sum $\sum_{v \in X_i} g^2(v)$ is nondecreasing in $i$ for
all $i \geq r$. Since this sum for $i=r$ is exactly
$\sum_{v \in L_2} g^2(v)$ and $t$ tends to infinity with
$m$, and since the sum over all $r \leq i \leq r+t$ is at most $1$,
it follows that the sum for $i=r$ is negligible.

\item Bounding the sum $\sum_{u \in L_1} g^2(u)$ is harder. 
Here we use the assumption that $|\mu | \geq (b+\varepsilon) \sqrt d$
where
$b=\frac{3d-1}{\sqrt {d(2d-1)}} $.
Define $X=X_0=\{u,u'\}$
and let $X_i$ denote the set of all vertices of $G$ of distance
exactly $i$ from $X$. Put $c=\sqrt{(2d-1)/d}$ and note that
$c+1/c=b$. Define a sequence of reals $s_0,s_1, s_2, \ldots $ as
follows. For $0 \leq i \leq r$, $s_i=c^id^{-i/2}$. For
$i=r+1$, $s_{r+1}=c^{r-1} d^{-(r+1)/2}$ and for all
$i \geq 1$, $s_{r+1+i}=\alpha_i d^{-(r+1+i)/2}$ where
the numbers $\alpha_i$ are defined by setting $\alpha_0=c^{r-1}$
and $\alpha_i/\alpha_{i-1}=x_i$ for $i \geq 1$ with
$x_1=1/c=\sqrt {d/(2d-1)}$ and for $i \geq 1$,
$$
x_{i+1}=\min \{b +\varepsilon-\frac{1}{x_i}, c\}.
$$
Using the sequence $s_i$ define a function $s$ on the vertices
in the union $\cup_{i \leq r+t} X_i$, where $r+t$ is smaller than
half the girth of $H$, by putting
$s(v)=s_i$ for all $v \in X_i$.
We proceed to show that for every vertex $v \in \cup_{i<r+t} X_i$,
\begin{equation}
\label{e93}
As(v) \leq |\mu| s(v). 
\end{equation}
For $v \in X=X_0$ this is equivalent to
$$
\frac{c}{\sqrt d} (d+1) \leq |\mu| 
$$
which is certainly true as
$$
|\mu| > b \sqrt d =(c+\frac{1}{c}) \sqrt d  \geq 
\frac{c}{\sqrt d} (d+1).
$$
For $v \in X_i$ with $1 \leq i \leq r-1$ the required inequality is
$$
\frac{\sqrt d}{c} s_i + \frac{d c s_i}{\sqrt d} \leq |\mu| s_i
$$
which follows from the fact that $1/c+c=b \leq b+\varepsilon$.
For $v \in X_r$ the inequality is
$$
\frac{2 \sqrt d}{c} s_i + \frac{d-1}{\sqrt d} \frac{s_i}{c} \leq
|\mu| s_i.
$$
For this it suffices to check that  
$$
\frac{2}{c}+\frac{d-1}{dc} \leq b+\varepsilon 
$$ 
which holds as the left hand side is equal to $b$.
For $v \in X_{r+1}$ it suffices  to check that
$$
\frac{\sqrt d s_{r+1}}{x_1} +\sqrt d x_2 s_{r+1} \leq |\mu| s_{r+1}
$$
which holds as $x_2 \leq b+\varepsilon-1/x_1$ by its definition.
Finally, for $v \in X_{r+1+i}, i \geq 1$ the required inequality
is equivalent to
$
\sqrt d \alpha_{i-1} + \sqrt d \alpha_{i+1} \leq |\mu| \alpha_i,
$
that is ,
$
\sqrt d \frac{1}{x_i} + \sqrt d x_{i+1} \leq (b+\varepsilon) \sqrt d
$
which again holds by the definition of $x_{i+1}$.
This completes the proof of (\ref{e93}). 
Next we observe that $x_1=1/c > 1/\sqrt 2 $, that $x_{i+1} \geq x_i$
for all $i$, and that if $x_i < c-\varepsilon$ then $x_{i+1} \geq
x_i + \varepsilon$.  Indeed $x_1 =\sqrt {d/(2d-1)} >1/\sqrt 2$, and
$x_{i+1} \leq c$ for all $i$ by the definition of $x_{i+1}$.  The
function $g(x)=b-1/x$ is increasing and concave in the interval 
$[1/c,c]$ and as $g(x)=x$ at the endpoints of the interval, $g(x)
 \geq x$  for all $x$ in the interval implying that 
$x_{i+1} =\min\{ g(x_i) +\varepsilon, c\} \geq x_i$ for all
$x_i \in [1/c,c]$ and that $x_{i+1} \geq x_i + \varepsilon$ if 
$x_i \leq c-\varepsilon$. 

By the above discussion it follows that $x_i$ is at least
$c > \sqrt{3/2}$ for all $i \geq 1/\varepsilon$ implying that 
the sequence $|X_j| s_j^2$ is increasing exponentially 
for $j > r+1/\varepsilon$. Between levels $r$ and $r+1/\varepsilon$
the terms of this sequence decrease by a factor larger than 
$1/c$ in each step, and it therefore follows that the term number
$r$ of this sequence is negligible compared to any term number
$r +\omega(1)$. This together with Lemma \ref{l33} completes the
proof.
\end{enumerate}
\end{proof}
Using the above ingredients we now verify all the claims in Theorem \ref{thm:totalloc}.
\begin{proof}[Proof of Theorem \ref{thm:totalloc}]
By the construction with described in Section \ref{construction},  there exists a graph $G=(V,E)$ with the following properties (with $r=\lfloor \alpha \log_d(m)\rfloor$):
\begin{enumerate}
\item $|V|=M=m+2n$ where $n=\frac{d^{r+1}+d^r-2}{d-1}=d^{r}\left(1+O(\frac{1}{d})\right).$
\item By Lemma \ref{girth}, and the choice of $r,$ we have $$\girth(G)\ge 2\log_{2d-1}((d+1) d^{r-1}) \ge \log_{2d-1}(d^{r})\ge 2\alpha\log_{2d-1}(m) = 2\alpha\log_d(m)\cdot (1-O(\log^{-1}(d))).$$
Thus the second claim in the theorem is verified by observing that $M=m+O(m^{\alpha}).$
\item By Proposition \ref{t31}, $$|\lambda_i(A_{M})|\le (3/\sqrt{2})\sqrt{d}$$ for all nontrivial eigenvalues $i\neq 1$ of the adjacency matrix $A_M$.
\item The third claim follows by Lemmas \ref{sym1}, \ref{lem:dense}, and \ref{global}, and noticing that $|V_1\cup V_2|=O(d^{r})=O(m^{\alpha}).$
\end{enumerate}
The final claim in the statement of the theorem is an easy consequence of Lemma \ref{lem:dense} which states that the set of finite $d-$ary tree eigenvalues is a dense subset of $(-2\sqrt{d}, 2\sqrt{d}).$ Furthermore, Remark \ref{partial} follows by choosing $S$ to be the top $\lfloor \varepsilon r \rfloor$ levels of $T_1.$
%
%
%
\end{proof}

We conclude by explaining how to modify the construction to produce many localized eigenvectors.
\begin{proof}[Proof of Theorem \ref{thm:many}] The proof follows from the observation that in the construction described in Section \ref{construction}, one can glue several trees to $H$ `far away' from each other to maintain high girth and every other property mentioned in Theorem  \ref{thm:totalloc}. More  precisely,  in the construction  in Section \ref{construction}, instead of considering a tree $T_1$ rooted at $u,$ consider trees $T^{(1)}_1,T^{(2)}_1,\ldots, T^{(k)}_1$ rooted at $u_1, u_2, \ldots, u_k$ and repeat the construction $k$ times with the corresponding trees $$\{T^{(1)}_2,T^{(2)}_2,\ldots, T^{(k)}_2\} \text{ and }  \{T^{(1)}_3, T^{(2)}_3,\ldots, T^{(k)}_3\}$$ rooted at $\{u'_1,u'_2,\ldots u'_k\}$ and $\{v_1,v_2,\ldots, v_k\}$ respectively. The same arguments as before imply that the girth condition in Theorem \ref{thm:totalloc} is satisfied as long as the graph distance in $H$ between any $u_i$ and $u_j$ is at least $4r.$ Lemma \ref{lem:pack} below shows that one can take $k$ to be as large as $m^{1-4\alpha}$ where $r=\lfloor \alpha \log_{d}(m)\rfloor.$  
Finally the proof of Proposition \ref{t31} in this case follows in exactly the same way after defining the set $X$ to be $\{v_1,v_2,\ldots, v_k\}$ and $\{u_1, u'_1, u_2, u'_2, \ldots, u_k, u'_k\}$ in 1. and 2. respectively instead of $\{v\}$ and $\{u,u'\}.$ 
\end{proof}

\begin{lemma}\label{lem:pack} Given any $d+1$-regular graph $G=(V,E)$ of size $m,$ 
and any $k,$ there are at least $\frac{m (d-1)}{(d+1)d^{k}}$ vertices in $V,$ all of whose mutual distances are at least $k.$ 
\end{lemma}
\begin{proof}Consider a maximal set $S$ of such vertices. Simple 
considerations imply that every other vertex in $V$ must be at 
distance at most $k$ from the set $S.$ Now the total number of
such vertices is at most $|S|{(d+1)d^{k}/(d-1)}$ which finishes the proof. 
\end{proof}

\bibliographystyle{alpha}
\bibliography{localized8.bbl}
%
%
%
%
%
	
\end{document}